\documentclass{amsart}
\usepackage{amsmath,amsfonts,mathrsfs}
\usepackage{paralist}
\usepackage{graphics} 
\usepackage{epsfig} 
\usepackage{graphicx} 
\usepackage{epstopdf}
\usepackage[colorlinks=true]{hyperref}
\hypersetup{urlcolor=blue, citecolor=red}
\newtheorem{theorem}{Theorem}[section]

\newtheorem{lemma}[theorem]{Lemma}

\theoremstyle{definition}
\newtheorem{definition}[theorem]{Definition}

\newcommand{\N}{\mathbb{N}}
\newcommand{\R}{\mathbb{R}}

\newcommand{\Leb}[1]{{\mathscr{L}}^{#1}} 

 \newcommand{\negint}{{\int\negthickspace\negthickspace\negthickspace\negthickspace-}}

\title[Logarithmic estimates for continuity equations]{Logarithmic estimates \\ for continuity equations}

\author[Maria Colombo, Gianluca Crippa and Stefano Spirito]{}



 \email{maria.colombo@sns.it}
 \email{gianluca.crippa@unibas.ch}
 \email{stefano.spirito@gssi.infn.it}

\begin{document}
\maketitle

\centerline{\scshape Maria Colombo}
\medskip
{\footnotesize
\centerline{Scuola Normale Superiore}
\centerline{Piazza dei Cavalieri 7}
\centerline{56126 Pisa, Italy}
} 

\medskip

\centerline{\scshape Gianluca Crippa}
\medskip
{\footnotesize
\centerline{Departement Mathematik und Informatik}
\centerline{Universit\"at Basel}
\centerline{Spiegelgasse 1}
 \centerline{CH-4051 Basel, Switzerland}
}

\medskip

\centerline{\scshape Stefano Spirito}
\medskip
{\footnotesize
\centerline{GSSI - Gran Sasso Science Institute}
\centerline{Viale Francesco Crispi 7}
 \centerline{67100 L'Aquila, Italy}
}

\bigskip

\begin{abstract}
The aim of this short note is twofold. First, we give a sketch of the proof of a recent result proved by the authors in the paper \cite{CCS} concerning existence and uniqueness of renormalized solutions of continuity equations with unbounded damping coefficient. Second, we show how the ideas in \cite{CCS} can be used to provide an alternative proof of the result in \cite{clop,des,bmo}, where the usual requirement of boundedness of the divergence of the vector field has been relaxed to various settings of exponentially integrable functions. 
\end{abstract}


\section{Introduction}\label{sec:Intro}
In this paper we consider the Cauchy problem for the continuity equation, namely
\begin{equation}\label{eqn:cont-damp}
\begin{cases}
\partial_t u(t,x)+ \nabla \cdot (b(t,x)u(t,x))= c(t,x)u(t,x)\\
u(0,x)=u_0(x)
\end{cases}
\end{equation}
where $(t,x)\in(0,T)\times\R^d$, $u\in\R$, $b\in\R^d$  and $c\in\R$. The continuity equation is a fundamental tool to study various nonlinear partial differential equations of the mathematical physics, for example equations arising in fluid mechanics and kinetic theory. In many physical situations the continuity equation has to be considered in a non-smooth setting. Starting from the papers of DiPerna and Lions \cite{lions} and Ambrosio \cite{ambrosio}, a huge literature has been developed in this direction (for an overview, see \cite{amcri-edi} and the references quoted therein). Roughly speaking, the continuity equation~\eqref{eqn:cont-damp} is well-posed in the class of bounded distributional solutions, and in the class of
renormalized solutions, if the vector field $b$ has bounded divergence, namely $\nabla\cdot b\in L^{1}(0,T;L^{\infty}(\R^d))$, and has at least a derivative of first order in some weak sense. More precisely, the case of Sobolev regularity has been considered in \cite{lions} and the  $BV$ regularity in \cite{ambrosio}. 
 
Concerning the source term $c$, that we will call damping term in analogy with fluid mechanics, the classical requirement is $c\in L^{1}(0,T;L^{\infty}(\R^d))$. In the paper \cite{CCS}  the case when the damping term is only in $ L^1((0,T)\times\R^d)$ has been considered. It turns out that dealing with this low integrability assumption requires a different approach from the ones in \cite{lions} and \cite{ambrosio}. 

Let us first explain at a very formal level the approach for the case where $c\in L^{1}(0,T;L^{\infty}(\R^d))$. 
In this case the crucial part in the theory is proving uniqueness of distributional solutions.
Let $u$ be the difference between two distributional solutions with the same initial datum. By linearity $u$ solves \eqref{eqn:cont-damp} with initial datum~$0$. Multiplying the equation by $2u$ one has
\begin{equation}
\label{eqn:dipl-strategy}
\begin{aligned}
\frac{d}{dt}\int_{\R^d} u(t,x)^2 \, dx 
&= \int_{\R^d} (2c(t,x)-\nabla \cdot b(t,x))  u(t,x)^2 \, dx
\\&
\leq
 (2\|c(t,\cdot)\|_{L^\infty(\R^d)}+ \|\nabla \cdot b(t,\cdot)\|_{L^\infty(\R^d)})\int_{\R^d} u(t,x)^2 \, dx.
\end{aligned}
\end{equation}
It follows by Gronwall Lemma that $\int_{\R^d} u(t,x)^2 \, dx =0$ for every $t\in [0,T]$, which implies uniqueness.

When $c$ is only in $L^1((0,T)\times\R^d)$ the previous formal calculation does not work anymore and we need to find another strategy to prove uniqueness. Inspired by the argument in \cite{crippade1} we proceed, always formally, as follows.
As in the computation \eqref{eqn:dipl-strategy}, we consider the difference $u$ of two solutions with the same initial datum and we multiply \eqref{eqn:cont-damp} by $u/(\delta+u^2)$, where $\delta>0$ is fixed, and we obtain
\begin{equation}
\label{eqn:ccs-strategy}
\begin{split}
\frac{d}{dt} \int_{\R^d} \! \log\Big(1+ \frac{ u(t,x)^2}{\delta} \Big) \, dx 
&= \int_{\R^d}\! \nabla \cdot b(t,x) \log\Big(1+ \frac{ u(t,x)^2}{\delta} \Big) \, dx\\
& \; \; + 2\int_{\R^d}\!(c(t,x)-\nabla\cdot b(t,x)) \frac{ u(t,x)^2}{\delta + u(t,x)^2}  \, dx\\
&\leq\|\nabla \cdot b(t,\cdot)\|_{L^\infty(\R^d)} \int_{\R^d} \log\Big(1+ \frac{ u(t,x)^2}{\delta} \Big) \, dx\\
& \; \; +2\int_{\R^d} |c(t,x)|+|\nabla\cdot b(t,x)| \, dx.
\end{split}
\end{equation}
(We have to assume here that the divergence is also globally integrable in space, however with a suitable truncation argument one can see that its boundedness is in fact sufficient.) By Gronwall Lemma we deduce that for every $t\in [0,T]$
\begin{equation*}
\begin{split}
 \int_{\R^d} \log\Big(1+ \frac{ u(t,x)^2}{\delta} \Big) \, dx 
&\leq
\exp\Big( \int_0^T \|\nabla \cdot b(t,\cdot)\|_{L^\infty(\R^d)} \, dt \Big)\\
&\hspace{2em}\cdot \int_0^T \int_{\R^d} 2(|c(t,x)|+|\nabla\cdot b(t,x)|) \, dx\, dt.
\end{split}
\end{equation*}
Letting finally $\delta$ go to $0$, since the right-hand side is finite and independent of $\delta$ we obtain that $u(t, \cdot ) = 0$. Now, we have a formal procedure to prove uniqueness which requires only the damping term $c$ to be integrable. However, a rigorous justification is definitely not straightforward and we will outline the main arguments in the next section. Then, in Section \ref{sec:bmo} we will show how the arguments can be adapted to the case of $\nabla\cdot b\in L^{1}(0,T; (BMO \cap L^1)(\R^d))$, i.e.~when the divergence is globally integrable and has bounded mean oscillation in the space variable (see Section~\ref{sec:bmo} for the definition of this space).

\section{The case of an integrable damping}\label{sec:2}
In this section we explain the main arguments in the paper \cite{CCS} to prove existence and uniqueness of solutions of the Cauchy problem \eqref{eqn:cont-damp}. The first nontrivial problem we have to face concerns the definition of solution. In the smooth setting there is a well-known representation formula for solutions of \eqref{eqn:cont-damp}. Let us recall that the continuity equation is strictly related to the ordinary differential equation for the flow of $b$, namely
\begin{equation}
\label{eqn:ODE}
\begin{cases}
\partial_t X(t,x) = b(t,X(t,x)) \qquad \forall t \in (0,T)
\\
X(0,x) =x
\end{cases}
\end{equation}
where $x\in \R^d$. Assuming that $u_0$, $b$ and $c$ are smooth and compactly supported, 
a solution of  \eqref{eqn:cont-damp} is then given in terms of the flow $X$ by the following well-known explicit formula: 
\begin{equation}\label{eqn:cont-damp-sol}
u(t,x)= \frac{u_0(X^{-1}(t,\cdot)(x))}{JX(t, X^{-1}(t,\cdot)(x))} \exp \Big({\int_0^t c(\tau, X(\tau, X^{-1}(t,\cdot)(x))) \, d\tau} \Big),
\end{equation}
where $JX(t,\cdot)$ denotes the Jacobian of the map $X(t,\cdot)$.

Denoting with $f_\sharp \mu$ the pushforward of a Borel measure $\mu$ on $\R^d$ through a Borel function $f:\R^d \to \R^d$,  \eqref{eqn:cont-damp-sol} can be equivalently rewritten as
\begin{equation}\label{eqn:cont-damp-sol-pushforw}
u(t,\cdot) \Leb{d} = X(t, \cdot)_\sharp \left( u_0 \exp\Big({\int_0^t c(\tau,X(\tau, \cdot)) \, d\tau}\Big) \Leb{d} \right).
\end{equation}
When $c$ is only integrable we cannot expect to have a solution $u$ in some Lebesgue space. Indeed, \eqref{eqn:cont-damp-sol} does not make sense as distributional solution even in the simplest autonomous cases: let $b(t,x) = 0$, $u_0 = 1_{[0,1]^d}$, and $c\in L^1(\R^d)$. A solution of \eqref{eqn:cont-damp} is given by $u(t,x) = u_0(x) e^{tc(x)}$; however $u(t,\cdot)$ may not belong to $L^1_{\rm loc}(\R^d)$ due to the low integrability of $c$.
In this case \eqref{eqn:cont-damp-sol} is not a distributional solution of \eqref{eqn:cont-damp}. Then, we are forced to replace distributional solutions by renormalized solutions (originally introduced in \cite{lions}),
whose definition does not require the function $u$ to be integrable. 
Specifically, the notion of solution for the Cauchy problem \eqref{eqn:cont-damp} is the following. 
\begin{definition}\label{defn:renorm-sol} Let $u_0: \R^d \to \R$ be a measurable  function, let $b\in L^1_{\rm loc} ((0,T) \times \R^d; \R^d)$ be a vector field such that $\nabla \cdot b\in L^1_{\rm loc} ((0,T) \times \R^d)$ and let $c\in L^1_{\rm loc} ((0,T) \times \R^d)$.
A measurable function $u:[0,T] \times \R^d \to \R$ is a renormalized solution of \eqref{eqn:cont-damp} with initial datum $u_0$ if for every function $\beta:\R\to \R$ satisfying
\begin{equation}
\label{defn:beta}
\beta \in C^1\cap L^\infty(\R), \qquad \beta'(z) z \in L^\infty(\R), \qquad \beta(0)=0
\end{equation}
 we have that
 \begin{equation}
 \label{defn:renorm-sol-eq}
 \partial_t \beta (u) + \nabla \cdot (b\beta(u)) + \nabla \cdot b \big(u \beta'(u)-\beta(u)\big) = cu \beta'(u)
 \end{equation} in the sense of distributions, namely 
  for every $\phi \in C^{\infty}_c ( [0,T) \times \R^d)$
 \begin{equation}
 \label{defn:renorm-sol-tested-st}
 \begin{split}
&\hspace{-3em}\int_{\R^d} \phi(0,x) \beta (u_0)\, dx +
 \int_0^T \! \int_{\R^d} [\partial_t \phi + \nabla\phi \cdot b ]\beta(u) \, dx \, dt 
  \\& + \int_0^T \!\! \int_{\R^d} \!\phi \Big[ \nabla \cdot b \big(\beta(u)-u \beta'(u)\big) +  cu \beta'(u) \Big] \, dx \, dt =0.
  \end{split}
\end{equation}
 \end{definition}
Notice that if we assume $c\in L^1( (0,T) \times \R^d)$ 
we have that 
\begin{equation}
\label{eqn:buona-def-JX}
\int_{\R^d} \int_0^T |c(\tau,X(\tau, x))| \, d\tau \, dx
\leq C
\int_0^T\int_{\R^d} |c(\tau,x)| \, d\tau \, dx <\infty
\end{equation}
since the flow is assumed to compress the Lebesgue measure in a controlled way, see~\cite{CCS}. Therefore the function $u$   in \eqref{eqn:cont-damp-sol} is well-defined pointwise almost everywhere.
 
In order to prove existence of renormalized solutions of \eqref{eqn:cont-damp} we argue as follows. First, we construct the flow $X$ for the ODE \eqref{eqn:ODE} by using results from \cite{lions,ambrosio,crippade1}, and then we check by a direct computation that \eqref{eqn:cont-damp-sol} is a renormalized solution of~\eqref{eqn:cont-damp}.
We remark that in the case of a divergence-free vector field $b$ proving that \eqref{eqn:cont-damp-sol} solves \eqref{eqn:cont-damp} is just a direct computation. On the other hand, the case when $b$ is not divergence-free requires a bit of work to justify the change of variable between Eulerian and Lagrangian coordinates. Precisely, some properties of the weak Jacobian of the flow $X$ have to be proven. We refer to Lemma 3.1 in \cite{CCS} for the details. 

After proving the existence of renormalized solutions we consider the problem of uniqueness. Reproducing the formal calculation explained in the Introduction in this non-smooth setting requires some work. First, as in \cite{lions}, we have to prove that when we consider the difference of two renormalized solutions, this is a renormalized solution of \eqref{eqn:cont-damp} with initial datum $0$.
This is nontrivial due to the nonlinear nature of the definition of renormalized solution. Precisely, in \cite{CCS} the following Lemma has been proved. 
\begin{lemma}[Lemma 4.2, \cite{CCS}]\label{lemma:diff-renorm}
Let $b\in L^1(0,T; BV_{\rm loc} (\R^d; \R^d))$ be a vector field with $\nabla \cdot b \in L^1(0,T; L^1_{\rm loc}(\R^d))$. Let $c\in L^1( (0,T) \times \R^d)$
and let $u_0:\R^d \to \R$ be a measurable function.
Let $u_1$ and $u_2$ be renormalized solutions of \eqref{eqn:cont-damp} with initial datum $u_0$.
Then $u:= u_1-u_2$ is a renormalized solution with initial datum $0$.
\end{lemma}
At this point we argue following the lines of the formal computation explained in the Introduction. 
We consider the function 
\begin{equation*}
\Gamma_{\delta,R}(t) = \int_{\R^d} \varphi_R(x) \beta_\delta (u(t,x))\, dx,
\end{equation*}
where for any $\delta>0$
\begin{equation}
\label{defn:beta-delta}
\beta_\delta(r) = \log\Big(1+\frac{[\arctan(r)]^2}{\delta} \Big) \qquad \forall r\in \R
\end{equation}
and for any $R>0$ the function $\varphi_R$ is smooth and has a suitable decay at infinity. 
Note that $\beta_\delta$ satisfies the hypothesis of the Definition \ref{defn:beta} and in particular 
\begin{equation}
\label{eqn:propr-gamma-de}
|r \beta'_\delta(r)| = \Big| \frac{2\arctan(r)}{\delta+ [\arctan(r)]^2} r\arctan'(r) \Big|
\leq 2 \qquad \forall r\in \R.
\end{equation} 
By taking the time derivative of $\Gamma_{\delta, R}(t)$ and using the definition of renormalized solution we get the following differential equality
 \begin{equation}
 \label{defn:renorm-sol-tested-2}
 \begin{split}
 \frac{d}{dt} \Gamma_{\delta,R}(t)
  = \int_{\R^d} \nabla\varphi_R \cdot b \, \beta_\delta(u) \, dx  +
\int_{\R^d} \varphi_R (c- \nabla \cdot b) u \beta_\delta'(u)  \, dx\\
+\int_{\R^d} \varphi_R \nabla \cdot b \, \beta_\delta(u) \, dx .
\end{split}
\end{equation}
By estimating carefully the right-hand side and using the fact that $\delta$ and $R$ are arbitrary we conclude, by using Gronwall Lemma, that $\Gamma_{\delta, R}$ is $0$. 
The rigorous statement of the result in \cite{CCS} is the following: 
\begin{theorem}\label{thm:main}
Let $b\in L^1(0,T; BV_{\rm loc} (\R^d; \R^d))$ be a vector field that satisfies a bound on the divergence 
$\nabla \cdot b \in L^1(0,T; L^\infty(\R^d))$ and the growth condition
\begin{equation}
\label{eqn:growth-b}
\frac{|b(x)|}{1+|x|} \in L^1(0,T; L^1(\R^d))+ L^1(0,T; L^\infty(\R^d)).
\end{equation}
Let 
$$c\in L^1( (0,T) \times \R^d)$$
and let $u_0:\R^d \to \R$ be a measurable function.
Then there exists a unique renormalized solution $u:[0,T] \times \R^d \to \R$ of \eqref{eqn:cont-damp} starting from $u_0$ and it is given by the formula 
\begin{equation}\label{eqn:cont-damp-sol-thm}
u(t,x)= \frac{u_0(X^{-1}(t,\cdot)(x))}{JX(t, X^{-1}(t,\cdot)(x))} \exp\Big({\int_0^t c(\tau, X(\tau, X^{-1}(t,\cdot)(x))) \, d\tau}\Big).
\end{equation}
\end{theorem}
\section{The case of divergence of $b$ in $BMO$}\label{sec:bmo}
In \cite{bmo} the author proved existence and uniqueness of solutions of the transport equation when the divergence $\nabla \cdot b$ of the vector field $b$ is the sum of a function in $L^\infty$ and a compactly supported function of bounded mean oscillation (defined in the sequel). This result extends the previous theory of \cite{lions,ambrosio} (see also \cite{amcri-edi} and the references quoted therein), where $\nabla \cdot b$ is assumed to be bounded in space.
Uniqueness is the most delicate point and its proof is based on a new inequality for $BMO$ functions, which gives the differential inequality $ \dot{\mathcal{E}}(t) \leq \mathcal{E}(t)|\log(\mathcal{E}(t))|$ in $[0,T]$, where $\mathcal{E}(t)$ is the $L^{2}$ norm of the difference of two solutions of the continuity equation \eqref{eqn:transp}. Then, uniqueness follows by Gronwall Lemma.

In this section we provide an alternative proof of this uniqueness result, in which we also somewhat refine the hypothesis on the compact support of the divergence, and allow general growth conditions on the vector field.

A similar result was proved previously in~\cite{des}. In such paper, the divergence of the vector field $b$ is assumed to be the sum of a bounded term and a term in an exponential space, namely 
\begin{equation}\label{eqn:div-expL}
\nabla \cdot b \in L^1(0,T; L^\infty(\R^d)) + L^1(0,T; {\rm Exp \,} L (\R^d)),
\end{equation}
where ${\rm Exp \,} L $ denotes the Orlicz space of globally exponentially integrable functions.
This condition on the divergence is more general than the one of \cite{bmo}, since compactly supported functions of bounded mean oscillation belong to the space ${\rm Exp \,} L (\R^d)$. However, the same proof of \cite{bmo}, as well as the one presented below, works when the divergence of $b$ satisfies the condition \eqref{eqn:div-expL}.

In a very recent paper~\cite{clop}, the authors improved the results in \cite{des,bmo}. They consider vector fields whose divergence is not necessarily compactly supported, and satisfies a weaker condition than~\eqref{eqn:div-expL}, expressed in terms of the Orlicz spaces $ {\rm Exp} \big( \frac{L}{\log^\gamma L} \big)$  for $\gamma \in (0,\infty)$. In particular, they show that if the vector field is locally Sobolev, satisfies the classical growth conditions, and
\begin{equation}
\label{eqn:Llog}
\nabla \cdot b \in L^1(0,T; L^\infty(\R^d)) + L^1\Big(0,T; {\rm Exp \,}\Big( \frac{L}{\log^\gamma L} \Big) (\R^d)\Big)
\end{equation}
for $\gamma=1$, then there is existence and uniqueness of bounded solutions of the continuity equation with given initial datum. The proof of their result is based on the differential inequality $ \dot{\mathcal{E}}(t) \leq \mathcal{E}(t) |\log(\mathcal{E}(t))| \log |\log(\mathcal{E}(t))|$ in $[0,T]$, which allows to apply Gronwall Lemma. 
They also adapt a counterexample of DiPerna and Lions~\cite{lions} to show that, for any exponent $\gamma>1$, the condition \eqref{eqn:Llog} is not enough to guarantee the uniqueness.

Before stating our result, we recall the definition and the properties of functions of bounded mean oscillation.
\begin{definition}
Given a locally integrable function $f :\R^d \to \R$ we consider its average on $B_r(x)$
$$(f)_{B_r(x)} = \negint_{B_r(x)} f(y) \, dy= \frac{1}{|B_r(x)|} \int_{B_r(x)} f(y) \, dy$$
and its mean oscillation in $B_r(x)$
$$\negint_{B_r(x)} |f(y) - (f)_{B_r(x)} | \, dy.$$
We say that $f$ is a function of bounded mean oscillation (BMO) if
\begin{equation}
\label{eqn:bmp-norm}
\sup_{r>0,\, x\in \R^d} \negint_{B_r(x)} |f(x) - (f)_{B_r(x)} | \, dx <\infty.
\end{equation}
\end{definition}
A natural norm $\| \cdot \|_{*}$ on the quotient space of $BMO$ functions modulo the space of constant functions is given by the quantity in \eqref{eqn:bmp-norm}.
From \cite[Lemma 1]{jn} we have that 
\begin{equation}\label{eqn:bmo-jonir2}
\Leb{d}\big( \{x\in K:  |f - ( f )_{K}| >r \} \big) \leq\frac{A}{\|f\|_{*}}\exp{(-\frac{br}{\|f\|_{*}})}\int_{K} |f - ( f )_{K}|\,dx
\end{equation}
for any $r>a\|f\|_{*}$, where $K$ is an arbitrary cube in $\R^d$ and $A, a, b$ are some universal constants depending only on the dimension $d$.   

In the following lemma we present the properties of $BMO$ functions which are used in the proof of Theorem \ref{thm:bmo}; in particular, we prove the exponential decay of the integral of $f$ on its superlevels.

\begin{lemma}\label{lemma:bmo-dec}
Let $f \in (BMO \cap L^1)(\R^d)$ be a nonnegative function. Then there exist $C, c>0$, depending only on $d$, such that for every $\lambda> a$
\begin{equation}
\label{eqn:bmo-est}
\int_{\R^d}{\big( f(x) - \lambda( \| f \|_{1}+\| f \|_{*}) \big)_+} \, dx \leq C \exp(-c\lambda)\| f \|_{1}.
\end{equation}
\end{lemma}
\begin{proof}
Let $K$ be any cube such that $\Leb{d}(K)>a^{-1}$.
Since the function $f$ is globally integrable, we have
\begin{equation}
\label{eqn:bmo-average}
( f)_{K} \leq \frac{ \|f\|_{1}}{\Leb{d}(K)} \leq  a{ \|f\|_{1}}\,.
\end{equation}
Thanks to \eqref{eqn:bmo-average}, for every $\lambda>a$  we have that $ \lambda \| f \|_{1}> (f)_{K}$, so that
$$\big( f(x) - \lambda( \| f \|_{1}+\| f \|_{*}) \big)_+ \leq \big( f(x) - (f)_{K} \big)_+ \qquad \forall x\in \R^d$$
 and similarly
$$\big\{ x\in K : f(x)> \lambda( \| f \|_{1}+\| f \|_{*}) \big\} \subseteq \big\{ x\in K : f(x)- (f)_{{K}}> \lambda \| f \|_{*} \big\}.$$
 Using also \eqref{eqn:bmo-jonir2}, we deduce that 
\begin{equation}
\label{eqn:bmo-sotto}
\begin{split}
&\hspace{-3em} \int_{K}{\big( f(x) - \lambda( \| f \|_{1}+\| f \|_{*}) \big)_+} \, dx \\
 &\leq 
 \int_{  \{x\in K \,:\, f- ( f)_{K}>  \lambda \| f \|_{*}\}
 }{\big( f(x) - ( f)_{K} \big)_+} \, dx 
 \\
 &\leq
 \int_{\lambda\| f \|_{*}}^\infty { \Leb{d}\big( \{ x\in K : f(x) - ( f)_{K}> r \} \big)}\, dr
 \\
 & \; \; +\lambda\|f\|_{*}\Leb{d}\big( \{ x\in K : f(x) - ( f)_{K}> \lambda\|f\|_{*} \} \big)
 \\
 &\leq
\frac{2A}{\|f\|_{*}}\|f\|_{1}\Big(\lambda\|f\|_{*}\exp(-b\lambda)+\int_{\lambda\|f\|_{*}}^{\infty}\exp(-\frac{br}{\|f\|_{*}})\,dr\Big)
 \\
 &=\frac{2A}{b}\|f\|_{1}(b\lambda+1)\exp(-b\lambda)
 \\
 &\leq
 C\|f\|_1\exp(-c\lambda)
 \end{split}
\end{equation}
for some constant $C,c$ depending only on the dimension $d$, with $c<b$. Since the cube $K$ is arbitrary, taking the supremum over all admissible $K$ we get \eqref{eqn:bmo-est}.
\end{proof}

In the following, we prove that the continuity equation
\begin{equation}
\label{eqn:transp}
\partial_t u+ \nabla \cdot( bu) = 0
\end{equation}
 with a $BV$ vector field with divergence in $(BMO\cap L^1)(\R^d)+L^\infty(\R^d)$ is well-posed in the class of bounded distributional solutions. We recall that the space $(BMO\cap L^1)(\R^d)$ is naturally endowed with the norm $\|\cdot\|_{1}+\|\cdot\|_{*}$ and that a  function $u \in L^\infty( (0,T) \times \R^d)$ is a distributional solution of~\eqref{eqn:transp} with initial datum $u_0 \in L^\infty(\R^d)$ if  for every $\phi \in C^{\infty}_c ( [0,T) \times \R^d)$
 \begin{equation*}
\int_{\R^d} \phi(0,x) u_0(x)\, dx +
 \int_0^T \! \int_{\R^d} [\partial_t \phi(t,x) + \nabla\phi (t,x) \cdot b(t,x) ]u(t,x) \, dx \, dt  \, =0.
\end{equation*}
 
\begin{theorem}\label{thm:bmo}
Let $u_0 \in L^\infty(\R^d)$ and $b\in L^1(0,T; BV_{\rm loc} (\R^d;\R^d))$ a vector field such that
\begin{equation}
\label{eqn:bmo-div-b}
\nabla \cdot b \in L^1(0,T; L^\infty(\R^d)) + L^1(0,T;(BMO\cap L^1)(\R^d)),
\end{equation}
and such that there exist two nonnegative functions $b_1$ and $b_2$ with
\begin{equation}\label{e:gro}
\frac{|b(t,x)|}{1+|x|} \leq b_1(t,x)+b_2(t),  \qquad b_1\in L^1(0,T; L^1(\R^d)), \quad b_2 \in L^1(0,T).
\end{equation}
Then there exists a unique distributional solution $u\in L^\infty( (0,T) \times \R^d)$ of \eqref{eqn:transp}
with initial datum $u_0$.
\end{theorem}
Existence is obtained through a standard regularization argument, see \cite{ambrosio}\break or~\cite[Appendix A2]{bmo}, and we omit the proof.
The result can be also generalized adding a right-hand side of the form $cu$, for some $c \in L^1((0,T)\times \R^d)$, with the same ideas as in Section~\ref{sec:2}, and we would have to consider renormalized solutions in place of distributional solutions.

The proof of uniqueness in Theorem~\ref{thm:bmo} starts by considering the difference of two distributional solutions with the same datum, which is a distributional (and hence renormalized, by \cite{ambrosio}) solution with initial datum $0$. As in the proof of Theorem~\ref{thm:main}, the choice of the particular renormalization function \eqref{defn:beta-delta}, together with a Lipschitz, decaying test function, allows to get a contradiction thanks to Gronwall Lemma.
Before proving the result, we state a lemma which allows to use a Lipschitz, decaying space function in place of smooth, compactly supported functions as a test function in the definition of renormalization.
\begin{lemma}\label{lemma:test-fund} 
Let $C>0$ and let  $b$ be as in Theorem~\ref{thm:bmo}. Let $u$ be a renormalized solution of \eqref{eqn:transp}
with initial datum $0$, let $\beta$ be a renormalization function satisfying~\eqref{defn:beta},
and let $\varphi \in W^{1,\infty}(\R^d)$ be a function such that
 \begin{equation}
 \label{defn:decay-test}
|\varphi(x)| \leq \frac{C}{(1+|x|)^{d+1}},
\qquad |\nabla \varphi(x)| \leq \frac{C}{(1+|x|)^{d+2}} \qquad a.e.\,\,\,\, x\in \R^d.
\end{equation}
Then the function $\int_{\R^d} \varphi(x) \beta (u(t,x))\, dx
$ coincides a.e. with an absolutely continuous function $\Gamma(t)$ such that
$\Gamma(0)=0 $
   and for a.e. $t\in [0,T]$
 \begin{equation}
 \label{defn:renorm-sol-tested-many-functions}
 \begin{split}
 \frac{d}{dt} \Gamma(t) 
 = \int_{\R^d} \nabla\varphi \cdot b\beta(u) \, dx  +
\int_{\R^d} \varphi \Big[ \nabla \cdot b \big(\beta(u)-u \beta'(u)\big) \Big] \, dx .
\end{split}
\end{equation}
\end{lemma}
\begin{proof}
The proof is a standard argument via approximation. Consider a sequence of smooth, compactly supported functions $\varphi_n$ satisfying the same decay \eqref{defn:decay-test} with $C$ independent on $n$ and approximating $\varphi$ strongly in $W^{1,1}_{\rm loc}$. 
By a well-known observation (see \cite[Remark 2.6]{CCS}), we can  test \eqref{defn:renorm-sol-eq} with the space function $\varphi_n$ to obtain that, for every $n\in \N$, the function $t \to \int_{\R^d} \varphi_n(x) \beta (u(t,x))\, dx $ coincides for a.e. $t \in [0,T]$ with an absolutely continuous function $\Gamma_n(t)$ which satisfies $\Gamma_n(0) = 0 $ and
 \begin{equation}
 \label{defn:renorm-sol-tested-many-functions-n}
 \frac{d}{dt} \Gamma_n(t) 
 = \int_{\R^d} \nabla\varphi_n \cdot b \, \beta(u) \, dx  +
\int_{\R^d} \varphi_n \Big[ \nabla \cdot b \big(\beta(u)-u \beta'(u)\big) \Big] \, dx .
\end{equation}
Thanks to  \eqref{defn:decay-test}, to the growth assumptions on $b$, and to the fact that $\nabla \cdot b$ is the sum of a bounded and an integrable function (by \eqref{eqn:bmo-div-b} and \eqref{eqn:bmo-average}), by dominated convergence the right-hand side of \eqref{defn:renorm-sol-tested-many-functions-n} converges to the right-hand side of \eqref{defn:renorm-sol-tested-many-functions} in $L^1(0,T)$.
Moreover by dominated convergence for a.e. $t\in [0,T]$
 \begin{equation}
 \label{defn:conv-gamma}
\Gamma(t) = \lim_{n\to \infty} \Gamma_n(t) = \lim_{n\to \infty} \int_{\R^d} \varphi_n(x) \beta (u(t,x))\, dx = \int_{\R^d} \varphi(x) \beta (u(t,x))\, dx.
\end{equation}
Hence the functions $\Gamma_n$ pointwise converge to the absolutely continuous function $\Gamma$, formula \eqref{defn:renorm-sol-tested-many-functions} holds, and $\Gamma (0)= 0$.
\end{proof}

\begin{proof}[Proof of uniqueness]
Given $R,\delta>0$, we consider the function $ \beta_\delta$
defined as in~\eqref{defn:beta-delta} and we define
\begin{equation}
\label{defn:phiR}
\varphi_R(x) =
\begin{cases}
\displaystyle{\frac{1}{2^{d+1}} }\qquad & x\in \R^d, \, |x|<R
\\
\displaystyle{\frac{R^{d+1}}{(R + |x|)^{d+1}} } \qquad & x\in \R^d, \, |x|>R
.
\end{cases}
\end{equation}
By the linearity of the continuity equation \eqref{eqn:transp}, up to taking the difference of two distributional solutions with the same initial datum, it is enough to show that any distributional solution $u$ with initial datum $0$ is constantly $0$. Thanks to the assumptions on $b$ (in particular, the local regularity $b\in L^1(0,T; BV_{\rm loc} (\R^d;\R^d))$ and the assumption on the divergence \eqref{eqn:bmo-div-b}, which implies that  $\nabla \cdot b \in L^1(0,T; L^1_{\rm loc}(\R^d))$) and to the results in \cite{ambrosio}, every bounded distributional solution of the continuity equation with  such a vector field $b$ is also renormalized.
In other words, $u$ is a renormalized solution with initial datum $0$ (in the sense of Definition~\ref{defn:renorm-sol} with $c=0$).
By Lemma~\ref{lemma:test-fund} we can use $\varphi_R$ as a test function in \eqref{defn:renorm-sol-tested-many-functions}; in other words, for every $R,\delta>0$ there exists an absolutely continuous function $\Gamma_{\delta,R}:[0,T] \to \R$ such that $\Gamma_{\delta,R}(0)=0$,
$$\Gamma_{\delta,R}(t) = \int_{\R^d} \varphi_R(x) \beta_\delta (u(t,x))\, dx \qquad \mbox{for a.e. }t \in [0,T],$$
 \begin{equation}
 \label{defn:bmo-renorm-sol}
 \begin{split}
 \frac{d}{dt} \Gamma_{\delta,R}
 &= \int_{\R^d} \nabla\varphi_R \cdot b \, \beta_\delta(u) \, dx +\int_{\R^d} \varphi_R \nabla \cdot b \big(\beta_\delta(u)-u \beta_\delta'(u)\big) \, dx \qquad \mbox{a.e. in } [0,T].
\end{split}
\end{equation}

We estimate the first term in the right-hand side of \eqref{defn:bmo-renorm-sol} as in \cite[(4.15)]{CCS}, thanks to the growth condition \eqref{eqn:growth-b} on $b$. Let $b_1$ and $b_2$ be the two nonnegative functions as in~\eqref{e:gro}.
For every $R>1$, by explicit computation, we have that 
$$|\nabla \varphi_R(x)|\leq {\bf 1}_{B^c_R}(x)(d+1) \varphi_R(x) (R+|x|)^{-1} \qquad \mbox{for every }x\in \R^d.$$
Therefore we obtain that for every $R>1$ and $t \in[0,T]$
\begin{equation}
\label{eqn:est-1}
\begin{split}
&\hspace{-3em}\int_{\R^d} \nabla\varphi_R \cdot b \, \beta_\delta(u) \, dx 
\leq
(d+1) \int_{\R^d \setminus B_R}\frac{\varphi_R}{R+|x|} (1+|x|) (b_1+b_2) \beta_\delta(u) \, dx 
\\
&\leq
(d+1) \int_{\R^d \setminus B_R}{\varphi_R} (b_1+b_2) \beta_\delta(u) \, dx 
\\
&\leq
(d+1) \log\Big(1+\frac{\pi^2}{4\delta}\Big) \int_{\R^d \setminus B_R} b_1\, dx 
+
(d+1) b_2 \int_{\R^d}{\varphi_R} \beta_\delta(u) \, dx .
\end{split}
\end{equation}
Let $d_1$ and $d_2$ be nonnegative functions such that
$$
\begin{aligned}
& |\nabla \cdot b(t,x)| \leq d_1(t,x)+d_2(t,x), \\ 
& d_1 \in L^1(0,T; L^\infty(\R^d)), \qquad
d_2 \in  L^1(0,T;(BMO\cap L^1)(\R^d)).
\end{aligned}
$$
Let $\lambda> a$ to be chosen later.
Since $|u \beta_\delta'(u)| \leq 2$ (see \eqref{eqn:propr-gamma-de}) and since $\varphi_R(x) \leq 2^{-d-1}$,  for every $\delta>0$ and $t\in[0,T]$ we estimate the second term in the right-hand side of \eqref{defn:bmo-renorm-sol}
\begin{equation}
\label{eqn:bmo-est-2}
\begin{aligned}
&\int_{\R^d} \varphi_R \nabla \cdot b \big(\beta_\delta(u)-u \beta_\delta'(u)\big) \, dx 
\leq \int_{\R^d} \varphi_R (d_1+d_2) \big|\beta_\delta(u)-u \beta_\delta'(u)\big| \, dx 
\\
&\leq 
(\|  d_1 \|_{\infty} + \lambda( \|  d_2 \|_{1}+ \|  d_2 \|_{*})) \int_{\R^d} \varphi_R \big(\beta_\delta(u)+2\big) \, dx\\ 
&\;\; +
\Big( \log\Big(1+\frac{\pi^2}{4\delta}\Big)+2\Big) 2^{-d-1} \int_{\R^d} \big(d_2  - \lambda  ( \|  d_2 \|_{1}+ \|  d_2 \|_{*})\big)_+ \, dx 
\\
&\leq 
(\|  d_1 \|_{\infty} + \lambda ( \|  d_2 \|_{1}+ \|  d_2 \|_{*})) \int_{\R^d} \varphi_R \big(\beta_\delta(u)+2\big) \, dx \\
&\;\; +
C \Big( \log\Big(1+\frac{\pi^2}{4\delta}\Big)+2\Big) \exp(-c\lambda) 2^{-d-1} \|d_2 \|_{1},
\end{aligned}
\end{equation}
where in the last inequality we applied Lemma~\ref{lemma:bmo-dec} to the function $|d_2(t, \cdot)|$ (notice that the norms appearing in the previous formula have to be intended in $\R^d$ for fixed $t$).

For every $t\in [0,T]$ we define the functions
$$
\begin{aligned}
a_\lambda(t) & =  \|  d_1(t, \cdot) \|_{L^\infty(\R^d)} + \lambda ( \|  d_2(t, \cdot) \|_{1}+ \|  d_2(t, \cdot) \|_{*})  + (d+1) b_2(t), \\
b_{\lambda, R}(t) & = 2\big(\|  d_1(t, \cdot) \|_{L^\infty(\R^d)} + \lambda  ( \|  d_2(t, \cdot) \|_{1}+ \|  d_2(t, \cdot) \|_{*} \big) \| \varphi_R \|_{L^1(\R^d)} 
\\&\hspace{1em}+ C 2^{-d}\exp(-c\lambda) \|d_2 (t,\cdot) 
 \|_{1} \\
c_{R}(t) & = (d+1) \|b_1(t,\cdot) \|_{L^1(\R^d \setminus B_R)}, \\
d_{\lambda}(t)
& = C2^{-d-1}\exp(-c\lambda) \|d_2(t,\cdot)\|_{1}.
\end{aligned}
$$
%
From \eqref{defn:bmo-renorm-sol}, \eqref{eqn:est-1}, and \eqref{eqn:bmo-est-2} we deduce that
 $$\frac{d}{dt}\Gamma_{\delta,R}(t)  \leq a_\lambda(t) \Gamma_{\delta,R}(t) + b_{\lambda,R}(t) + (c_R(t)+ d_{\lambda}(t) ) \log\Big(1+\frac{\pi^2}{4\delta}\Big).$$
 Let $\tau_0>0$ to be chosen later in terms of $b$ and independent on $R, \lambda$. It is enough to show that $u(t,\cdot) = 0$ for every $t\in [0, \tau_0]$, then we can repeat the argument in the subsequent time intervals.
Since by assumption $\Gamma_{\delta,R}(0)=0$, using Gronwall Lemma we obtain that for every $t\in [0,\tau_0]$
\begin{equation}
\label{est:bmo-gronw}
\begin{split}
\Gamma_{\delta,R}(t)  
&\leq \exp\Big(\int_0^{\tau_0} a_\lambda (s)ds\Big)\Big(\int_0^{\tau_0} b_{\lambda, R}(s)ds + \log\Big(1+\frac{\pi^2}{4\delta}\Big) \int_0^{\tau_0} (c_{R}(s) + d_{\lambda}
) ds \Big)
\\
&=\exp(A_\lambda )\Big(B_{\lambda, R} + \log\Big(1+\frac{\pi^2}{4\delta}\Big)(C_{R} + D_{\lambda}) \Big).
\end{split}
\end{equation}

If, by contradiction, $u(t,\cdot)$ is not identically $0$ for some $t\in [0,{\tau_0}]$, then there exist $R_0>0$ and $\gamma>0$ such that $m=\Leb{d}(\{ x \in B_{R_0}: [\arctan u(t,x)]^2> \gamma\})>0$.
We obtain that for every $R\geq R_0$
\begin{equation}
\label{eqn:bmo-contr}
\begin{split}
\frac{m }{2^{d+1}}\log\Big(1+\frac{\gamma}{\delta} \Big) 
&\leq \Gamma_{\delta,R}(t) \leq \exp(A_\lambda)\Big(B_{\lambda,R} + \log\Big(1+\frac{\pi^2}{4\delta}\Big)(C_{R} + D_{\lambda}) \Big).
\end{split}
\end{equation}
First, we fix $\tau_0>0$ so that 
$$\int_0^{\tau_0}  ( \|  d_2(s, \cdot) \|_{1}+ \|  d_2(s, \cdot) \|_{*}) \, ds \leq \frac{c}{2}\,.$$
Then, for a constant $C>0$, we have 
$$\exp({A_\lambda})D_{\lambda}\leq C\exp\Big(-\frac{c}{2}\lambda\Big),$$
and 
we can choose $\lambda$ big enough to have 
$$\exp(A_{\lambda})D_{\lambda}\leq \frac{m }{2^{d+3}}.$$ 
Since $b_1\in L^{1}((0,T)\times\R^d)$ we can choose $R$ big enough to have 
$$\exp(A_\lambda) C_R\leq \frac{m }{2^{d+3}}.$$
Then, we are left with 
\begin{equation}\label{eqn:final}
\frac{m }{2^{d+1}}\log\Big(1+\frac{\gamma}{\delta} \Big)\leq \exp(A_\lambda)B_{\lambda,R}+\frac{m }{2^{d+2}} \log\Big(1+\frac{\pi^2}{4\delta}\Big).
\end{equation}
Dividing  \eqref{eqn:final}  by $\log(\delta^{-1})$ and letting $\delta$ go to $0$ 
we find a contradiction.
\end{proof}

\section*{Acknowledgments} This research has been partially supported by the SNSF grants 140232 and 156112 and has been started while the first and third author were a visitor and a PostDoc, respectively, at the Departement Mathematik und Informatik of the Universit\"at Basel. They would like to thank the department for the hospitality and the support. This note has been written on the occasion of the conference ``Contemporary topics in conservation laws'' in Besan\c con (France), February 9--12, 2015. The second author thanks the organizers for the kind invitation to the conference. The authors wish to thank the anonymous referees for the useful comments, which led to a cleaner statement of the main result.

\medskip
Received xxxx 20xx; revised xxxx 20xx.
\medskip

\end{document}